\documentclass[11pt,twoside,reqno,centertags,draft]{amsart}
\usepackage{amsfonts}
\usepackage{color,enumitem,graphicx}
\usepackage[colorlinks=true,urlcolor=blue,
citecolor=red,linkcolor=blue,linktocpage,pdfpagelabels,
bookmarksnumbered,bookmarksopen]{hyperref}

  \usepackage{amsmath,amsthm,amsfonts,amssymb}

\begin{document}
\title{ Classification of isolated singularities of positive solutions for Choquard equations}
\date{}
\maketitle

\vspace{ -1\baselineskip}

{\small
\begin{center}

\medskip

  {\sc  Huyuan Chen}
  \medskip

 Department of Mathematics, Jiangxi Normal University,\\
Nanchang, Jiangxi 330022, PR China\\[2mm]
and\\
 NYU-ECNU Institute of Mathematical Sciences at NYU Shanghai,\\
 Shanghai 200120, PR China\\
 Email: chenhuyuan@yeah.net\\[10pt]
 {\sc  Feng Zhou}
   \medskip

Center for PDEs and Department of Mathematics, East China Normal University,\\
 Shanghai, PR China\\
 Email: fzhou@math.ecnu.edu.cn\\[10pt]

\end{center}
}

\medskip

\begin{quote}
{\bf Abstract.}
In this paper we classify the isolated singularities of positive solutions to Choquard equation
 $$\displaystyle \ \  -\Delta   u+  u =I_\alpha[u^p] u^q\;\;
  {\rm in}\;  \mathbb{R}^N\setminus\{0\},
 \;\;  \displaystyle   \lim_{|x|\to+\infty}u(x)=0,
$$
where $p >0, q \geq 1, N \geq 3, \alpha \in (0,N)$ and
$I_{\alpha}[u^p](x) = \int_{\mathbb{R}^N} \frac{u(y)^p}{|x-y|^{N-\alpha}}dy.$
We show that any positive solution  $u$
is a solution of
 $$ -\Delta u+u=I_\alpha[u^p]u^q+k\delta_0\quad {\rm in}\; \mathbb{R}^N$$
in the distributional sense for some $k\ge 0$, where $\delta_0$ is the Dirac mass at the origin.
We prove the existence of singular solutions in the subcritical case: $ p+q< \frac{N+\alpha}{N-2}\;{\rm and}\; p< \frac{N}{N-2}, q<\frac{N}{N-2}$
 and prove that either the solution $u$ has removable singularity at the origin
or satisfies $\lim_{|x| \to 0^+} u(x)|x|^{N-2}=C_N$ which is a positive constant. In the supercritical case: $p+q\ge \frac{N+\alpha}{N-2}\;{\rm or}\; p\ge \frac{N}{N-2},\;{\rm or}\; q\ge \frac{N}{N-2}$ we prove that $k=0$.

\end{quote}

 \renewcommand{\thefootnote}{}

\newcommand{\N}{\mathbb{N}}
\newcommand{\R}{\mathbb{R}}
\newcommand{\Z}{\mathbb{Z}}

\newcommand{\cA}{{\mathcal A}}
\newcommand{\cB}{{\mathcal B}}
\newcommand{\cC}{{\mathcal C}}
\newcommand{\cD}{{\mathcal D}}
\newcommand{\cE}{{\mathcal E}}
\newcommand{\cF}{{\mathcal F}}
\newcommand{\cG}{{\mathcal G}}
\newcommand{\cH}{{\mathcal H}}
\newcommand{\cI}{{\mathcal I}}
\newcommand{\cJ}{{\mathcal J}}
\newcommand{\cK}{{\mathcal K}}
\newcommand{\cL}{{\mathcal L}}
\newcommand{\cM}{{\mathcal M}}
\newcommand{\cN}{{\mathcal N}}
\newcommand{\cO}{{\mathcal O}}
\newcommand{\cP}{{\mathcal P}}
\newcommand{\cQ}{{\mathcal Q}}
\newcommand{\cR}{{\mathcal R}}
\newcommand{\cS}{{\mathcal S}}
\newcommand{\cT}{{\mathcal T}}
\newcommand{\cU}{{\mathcal U}}
\newcommand{\cV}{{\mathcal V}}
\newcommand{\cW}{{\mathcal W}}
\newcommand{\cX}{{\mathcal X}}
\newcommand{\cY}{{\mathcal Y}}
\newcommand{\cZ}{{\mathcal Z}}

\newcommand{\abs}[1]{\lvert#1\rvert}
\newcommand{\xabs}[1]{\left\lvert#1\right\rvert}
\newcommand{\norm}[1]{\lVert#1\rVert}

\newcommand{\loc}{\mathrm{loc}}
\newcommand{\p}{\partial}
\newcommand{\h}{\hskip 5mm}
\newcommand{\ti}{\widetilde}
\newcommand{\D}{\Delta}
\newcommand{\e}{\epsilon}
\newcommand{\bs}{\backslash}
\newcommand{\ep}{\emptyset}
\newcommand{\su}{\subset}
\newcommand{\ds}{\displaystyle}
\newcommand{\ld}{\lambda}
\newcommand{\vp}{\varphi}
\newcommand{\wpp}{W_0^{1,\ p}(\Omega)}
\newcommand{\ino}{\int_\Omega}
\newcommand{\bo}{\overline{\Omega}}
\newcommand{\ccc}{\cC_0^1(\bo)}
\newcommand{\iii}{\opint_{D_1}D_i}

\numberwithin{equation}{section}

\vskip 0.2cm \arraycolsep1.5pt
\newtheorem{lemma}{Lemma}[section]
\newtheorem{theorem}{Theorem}[section]
\newtheorem{definition}{Definition}[section]
\newtheorem{proposition}{Proposition}[section]
\newtheorem{remark}{Remark}[section]
\newtheorem{corollary}{Corollary}[section]

\setcounter{equation}{0}
\section{Introduction}

In this paper, we are concerned with the classification of all positive solutions to Choquard equation
\begin{equation}\label{eq 1.1}
  \arraycolsep=1pt
\begin{array}{lll}
 \displaystyle  \ \ -\Delta   u+  u =I_\alpha[u^p] u^q\quad
  {\rm in}\quad  \mathbb{R}^N\setminus\{0\},\\[2mm]
 \phantom{    }
\qquad \ \displaystyle   \lim_{|x|\to+\infty}u(x)=0,
\end{array}
\end{equation}
where     $p>0,\, q\ge1$,    $N\ge3$, $\alpha\in(0,N)$ and
$$I_\alpha[u^p](x)=\int_{\mathbb{R}^N}\frac{u(y)^p}{|x-y|^{N-\alpha}}\, dy.$$

When $N=3$, $\alpha=p=2$ and $q=1$, problem (\ref{eq 1.1}) was proposed  by P. Choquard as an approximation to Hartree-Fock theory for a one component plasma,
which has been explained in Lieb and Lieb-Simon's papers \cite{L0,LS} respectively. It is also called Choquard-Pekar equation after a more early work of S. Paker for describing the quantum mechanics of a polaron at rest \cite{P}, or sometime the nonlinear Schr\"odinger-Newton equation in the context of self-gravitating matter \cite{WW}. The Choquard type equations also arise in the  physics of multiple-particle systems,
see \cite{G}. Furthermore, the Choquard type equations appear to be a prototype of the nonlocal problems, which play a fundamental role in some Quantum-mechanical and non-linear optics,
refer to \cite{GL,O}. When $\alpha\in(0,2)$, the Riesz potential $I_\alpha$ is related to the fractional Laplacian, which is a nonlocal operator,
so the Choquard equation (\ref{eq 1.1}) could be divided into a system with the Laplacian in the
linear part of the first equation and fractional Laplacian in the second one. For the related topics on the fractional equation we can refer for example to \cite{CS0,CS1,CV0,CV1}.

The study of isolated singularities is initiated by Brezis and Lions in \cite{BL}, where an useful tool to connect the singular solutions of
elliptic equation in punctured domain and the solutions of corresponding elliptic equation in the  distributional
 sense was built, by the study of
$$  \arraycolsep=1pt
\begin{array}{lll}
   \Delta    u\le  a u+f\quad
  {\rm in}\; \Omega\setminus\{0\},\quad u>0\;\; {\rm in} \; \Omega\setminus\{0\},\\[2mm]
 \quad \ u\in L^1(\Omega), \quad \Delta u\in L^1(\Omega\setminus\{0\}),
\end{array}
$$
where $\Omega$ is a bounded domain in $\R^N $ containing the origin, the parameter $a>0$ and function $f\in L^1(\Omega)$.
Later on, the classification of isolated singular problem
\begin{equation}\label{eq 1.4}
\arraycolsep=1pt
\begin{array}{lll}
 -\Delta   u=u^p\quad
 &{\rm in}\quad \Omega\setminus\{0\},
 \\[2mm]
 \phantom{  --   }
 \displaystyle u>0\ \
 &{\rm in}\ \ \Omega
\end{array}
\end{equation}
  was performed  by Lions in \cite{L} for $p\in (1,\frac{N}{N-2})$,
by  Aviles in \cite{A} for $p = \frac{N}{N-2}$, by Gidas-Spruck in \cite{GS}
for $\frac{N}{N-2} < p < \frac{N+2}{N-2}$,  by Caffarelli-Gidas-Spruck in \cite{CGS} for
$p=\frac{N+2}{N-2}$. For the case that $p>\frac{N+2}{N-2}$, the classification  of isolated singularities
for (\ref{eq 1.4}) is still open.  In the particular case of $p\in (1,\frac{N}{N-2})$, any positive solution of (\ref{eq 1.4})
is a solution of
  \begin{equation}\label{eq 1.5}
 \displaystyle    -\Delta    u=u^p+k\delta_0\quad{\rm in}\quad \Omega
\end{equation}
in the distributional
 sense for some $k\ge0$. Furthermore, for suitable $k$, problem \eqref{eq 1.5} has at least two positive solutions including the minimal solution. More related topics could be referred to the 
references  \cite{BP,BB,BV,
DDG,DGW,MV,V}.

Our interest in this paper is to classify the singularities of positive classical solutions for Choquard   equation (\ref{eq 1.1}). Here $u$ is said to be a classical solution of (\ref{eq 1.1}) if $u\in C^2(\R^N\setminus\{0\})$, $I_\alpha[u^p]$ is well-defined in $\R^N\setminus\{0\}$ and $u$ satisfies  (\ref{eq 1.1}) pointwisely. The first result can be stated  as follows.


\begin{theorem}\label{teo 0}
Assume that $N\ge3$,  $\alpha\in(0,N)$,
 $p>0$,  $ q\ge 1$
 and $u$ is a positive classical solution of (\ref{eq 1.1}) satisfying $u\in L^p(\R^N)$.

     Then  $I_\alpha[u^p]u^q\in L^1(\R^N)$ and
there exists $k\ge0$
  such that  $u$ is a solution of
 \begin{equation}\label{eq 1.2}
      -\Delta u+u=I_\alpha[u^p]u^q+k\delta_0\quad
 {\rm in}\quad \R^N,
\end{equation}
in the distributional sense, that is the following identity holds,
 \begin{equation}\label{1.1}
     \int_{\R^N} [u (-\Delta  \xi+ \xi) -I_\alpha[u^p]u^q\xi]\, dx=k\xi(0),\quad \forall \xi\in C^\infty_c(\R^N),
\end{equation}
where $C^\infty_c(\R^N)$ is the space of all the functions in  $C^\infty(\R^N)$   with  compact support.

Furthermore,
$(i)$  when
\begin{equation}\label{1.3}
 p+q\ge \frac{N+\alpha}{N-2}\quad{\rm or}\quad p\ge \frac{N}{N-2}\quad{\rm or}\quad q\ge \frac{N}{N-2},
\end{equation}
then $k=0$.

$(ii)$ When
\begin{equation}\label{1.4}
 p+q< \frac{N+\alpha}{N-2}\quad{\rm and}\quad p< \frac{N}{N-2}\quad{\rm and}\quad q<\frac{N}{N-2}£¬
\end{equation}
  and if $k=0$,  then $u$ is a classical solution of
  \begin{equation}\label{eq 1.3}
  \arraycolsep=1pt
\begin{array}{lll}
 \displaystyle  -\Delta  u+u=I_\alpha[u^p]u^q\quad
 &{\rm in}\quad \R^N,\\[2mm]
 \phantom{   }
 \qquad \displaystyle   \lim_{|x|\to+\infty}u(x)=0;
\end{array}
\end{equation}
 if $k>0$, then $u$ satisfies that
\begin{equation}\label{1.2}
 \lim_{|x|\to0^+} u(x)|x|^{N-2}=c_{N} k,
\end{equation}
where $c_N$ is the normalized constant.
\end{theorem}

The solution of (\ref{eq 1.1}) in the distributional sense are sometimes called the very weak solution. We call also the pair exponent $(p,\,q)$ is supercritical if (\ref{1.3}) holds  and  $(p,\,q)$ is subcritical if (\ref{1.4}) does. Theorem \ref{teo 0} shows that in the supercritical case,
the singularities of positive solutions of (\ref{eq 1.1}) are not visible in the distribution sense by the Dirac mass. 
In the subcritical case the solutions of (\ref{eq 1.1}) may have the singularity as $|x|^{2-N}$ or removable singularity at the origin.  
In the subcritical case and when $k=0$, we improve the regularity of $u$ by separating the factors $I_\alpha[u^p],\, u^q$ of nonlinearity and using  the bootstrap argument,
however, the factors $I_\alpha[u^p],\, u^q$ have different growth rates in $L^t$ estimates and the key point is to balance them; while $k>0$, in order to study (\ref{1.2}),
our strategy is to divide $u$ into
$$u\le u_n+\sum_{i=1}^{n-1}\Gamma_i+ k\Gamma_0,$$
where $\Gamma_0$ is the fundamental solution of $-\Delta u+u=\delta_0$ in $\R^N$,
$\Gamma_i$ are generated by $\Gamma_0$ but with lower singularities, $u_n$ is the remainder term.
Our aim here is to find some $n_0$ such that $u_{n_0}$ is bounded at the origin. The difficulty in this procedure
is   to control the singularity of $\sum_{i=1}^{n-1}\Gamma_i$ and to improve the regularity of $u_n$ at the same time.
To this end, we develop the bootstrap argument, to reduce the singularity of $\Gamma_n$ first  until to be bounded and then to improve the regularity of $u_n$ without the influence
of singularities from  $\Gamma_n$.
We mention that \cite{GL,L1,MVS} show that the nonlinear Choquard equation 
 admits variational solutions in the case of $q =p-1$, which have no singularities at the origin.
  See a survey \cite{MVS2} and the references therein.



Our second aim  of this paper is to decide  whether (\ref{eq 1.1}) has singular solutions in the subcritical case. To this end, we shall search the weak solutions of (\ref{eq 1.2}), where  the restriction $\lim_{|x|\to +\infty} u(x)=0$
in (\ref{eq 1.2}) is viewed as
$$\lim_{r\to+\infty} {ess{\rm sup}_{x\in\R^N\setminus B_r(0)}}u(x)=0.$$
Now it is ready to state the existence and nonexistence of weak solutions of (\ref{eq 1.2}).

 \begin{theorem}\label{teo 1}
Assume that  $N\ge3$,  $\alpha\in(0,N)$, $p>0,\, q\ge1$ satisfy (\ref{1.4}) and denote
$$k_q=\left(\frac1{c_{1} (p+q)}\right)^{\frac1{p+q-1}}\frac{p+q-1}{p+q},$$
where $c_1$ is the constant from (\ref{1.5}).
Then there exists $k^*\ge k_q$
  such that

  $(i)$ for $ k\in(0, k^*)$, problem (\ref{eq 1.2}) admits a minimal positive weak solution $u_{k}$;

 $(ii)$ for $k> k^*$,  problem (\ref{eq 1.2}) admits no positive weak solution.

Furthermore, if $k\le k_q$, $u_k$ is a classical solution of (\ref{eq 1.1}) and satisfies (\ref{1.2}).

\end{theorem}
When $q=p-1$, V. Moroz and J. Van Schaftingen \cite{MVS1} have derived  groundstates of (\ref{eq 1.3}) for $\frac{N+\alpha}{N}<p<\frac{N+\alpha}{N-2}$
by variational method. Furthermore, this existence result is sharp in the sense that there is no nontrivial regular variational solution to  (\ref{eq 1.3}) for
$p\le \frac{N+\alpha}{N}$ and $p\ge \frac{N+\alpha}{N-2}$. Usually, the derivation of the very weak solution is different by using nonvariational methods.
The solution $u_k$ of (\ref{eq 1.2}) is derived by iterating procedure:
$$v_0= k \mathbb{G} [\delta_0], \qquad v_n  = \mathbb{G} [v_{n-1}^p]+ k \mathbb{G} [\delta_0],$$
where $\mathbb{G}$ is the Green's operator defined by the Green kernel $G(x,y)$ of $  -\Delta+id $ in $\R^N\times\R^N$.
Here the main difficulty  is to find a barrier function to control the sequence $\{v_n\}_n$. It is well-known that
in the bounded domain $\Omega$ and 
$\gamma\in(1,\frac{N}{N-2})$ 
the barrier function is constructed by the fact that
 \begin{equation}\label{1.5q}
 \mathbb{G}_\Omega[\mathbb{G}_\Omega^\gamma[\delta_0]]\le c_2\mathbb{G}_\Omega[\delta_0]\quad {\rm in}\quad \Omega\setminus\{0\},
 \end{equation}
where $c_2>0$ and $\mathbb{G}_\Omega[\cdot]$ is the Green's operator defined by Green's kernel  of $  -\Delta+id $  in $\Omega\times\Omega$.
However, the estimate  (\ref{1.5q}) is no longer valid for  $\mathbb{G}$. In order to find a barrier function when the domain is the whole space,
our strategy here is to establish the following estimate
 \begin{equation}\label{1.5}
 \mathbb{G}[I_\alpha[\Phi_0^p]\Phi_0^q]\le c_1\Phi_0\quad {\rm in}\quad \R^N\setminus\{0\},
 \end{equation}
where $c_1>0$ and  $\Phi_0$ satisfies that
$$-\Delta u+\frac14 u=\delta_0\quad{\rm in}\quad \R^N.$$

Recently, M.Ghergu and S.D. Taliaferro \cite{GT} have studied the behavior near the origin in $\R^n$
for the Choquard-Pekar type inequality
\begin{equation}\label{1.9}
0\le -\Delta u\le (|x|^{-\alpha} \ast u^\lambda) u^\sigma \quad{\rm in}\quad B_2(0)\setminus \{0\}.
\end{equation}
Here $u$ is assumed to be in $C^2(\R^n \setminus \{0\}) \cap L^{\lambda}(\R^n)$ and
 $\ast$ is the convolution operator. In particular, they proved that for some suitable range of $\lambda, \sigma$ depending on $n$ and $\alpha$, the existence of pointwise bounds for nonnegative solutions of (\ref{1.9}). We mention that the nonnegative solutions they considered are superharmonic functions, and the operator $-\Delta + id$ in our case make a great difference on the analysis of the singularities and the existence of singular solutions.

We emphasize that in this paper we consider the case where $q \geq 1$.
When $q<1$,  \cite{MVS,MVS1} show that the solutions of problem (\ref{eq 1.1}) may have polynomial decay at infinity,  which makes the classification of singularities  of (\ref{eq 1.1}) difficult and interesting. In fact, the polynomial can not guarantee that $I_\alpha[u^p] $ is well defined and then it may cause the nonexistence. The existence and nonexistence of isolated singular solution of (\ref{1.1}) when $q \in (0,1)$ is considered in
\cite{CZ}.

The rest of this paper is organized as follows. In Section 2, we  show the integrability of the solutions for equation (\ref{eq 1.1}) and the singularity of the functions generated by
the fundamental solution of $-\Delta u+u=\delta_0$ in $\R^N$.   Section 3 is devoted
to the classification of the singularities of positive solutions for (\ref{eq 1.1}) and in Section 4, we search the weak solutions of (\ref{eq 1.2}) in the subcritical case.
\vskip0.5cm

 \section{ Preliminary }

We start the analysis from the integrability of the solutions to (\ref{eq 1.1}) near the origin. In what follows,  we denote by $c_i$ a generic  positive constant.

 \begin{lemma}\label{lm 2.2}
 Assume that $N\ge3$, $\alpha\in(0,N)$, $p>0,$ $q>0$ and $u$ is a positive classical solution of (\ref{eq 1.1})
 such that $u\in L^p(\R^N)$.
  Then we have 
\begin{equation}\label{lp}
 I_\alpha[u^p]u^q\in L^1_{loc}(\R^N).
\end{equation}

 \end{lemma}
\begin{proof}
 If $I_\alpha[u^p]u^q\not\in L^1_{loc}(\R^N)$, then
$$\lim_{r\to0^+}\int_{B_1(0)\setminus B_r(0)} I_\alpha[u^p]u^q\,dx=+\infty$$
by the facts that $u\ge 0$ and   $u\in L^\infty_{loc}(\R^N\setminus\{0\})$.
Thus there exists a decreasing sequence $\{R_n\}_n\subset(0,1)$ such that
$\lim_{n\to\infty}R_n=0$ and
\begin{equation}\label{2.0}
 \int_{B_1(0)\setminus B_{R_n}(0)} I_\alpha[u^p]u^q\,dx=n.
\end{equation}
Let $w_n$ be the solution of
$$
  \arraycolsep=1pt
\begin{array}{lll}
  \displaystyle   -\Delta    w_n+w_n=\chi_{B_1(0)\setminus B_{R_n}(0)}I_\alpha[u^p]u^q\quad
 &{\rm in}\quad \R^N,\\[2mm]
 \phantom{ \ }
 \displaystyle   \lim_{|x|\to+\infty}w_n(x)=0,
\end{array}
$$
where $\chi_{D}$ denotes the standard characteristic function of a domain $D$.
Let $\Gamma_0$ be the fundamental solution of $-\Delta+id$, then
$$\lim_{|x|\to0^+} (u+\Gamma_0)(x)=+\infty\quad {\rm and}\quad \lim_{|x|\to+\infty} (u+\Gamma_0)(x)=0,$$
so it follows by Comparison Principle that for any $n\in\N$,
\begin{equation}\label{2.1}
u+\Gamma_0\ge w_n\quad{\rm in}\ \ \R^N\setminus\{0\}.
\end{equation}
Note that   $G(x,y)\ge c_3$ for any $x,y\in B_1(0)$, then by (\ref{2.0}), we have that
\begin{eqnarray*}
w_n(x) &=& \mathbb{G}[\chi_{B_1(0)\setminus B_{R_n}(0)} I_\alpha[u^p]u^q ](x)=\int_{B_1(0)\setminus B_{R_n}(0)} G(x,y) I_\alpha[u^p](y)u(y)^qdy \\
  &\ge&  c_3\int_{ B_1(0)\setminus B_{R_n}(0)} I_\alpha[u^p](y)u(y)^q dy\ \ = c_3n\\[1.5mm]
  &\to&+\infty \quad{\rm as}\ \ n\to\infty,\quad \ \forall x\in B_1(0),
\end{eqnarray*}
which, together with (\ref{2.1}),  implies that
$u+\Gamma_0\equiv+\infty$ in $B_1(0)$ and this is impossible.
Therefore we have that $ I_\alpha[u^p]u^q\in L^1_{loc}(\R^N)$.
\end{proof}

The following asymptotic behavior of  positive solutions to  problem (\ref{eq 1.1}) plays an important role
in the control of the integrability at infinity.

 \begin{lemma}\label{lm 2.1}
 Assume that $N\ge3$, $\alpha\in(0,N)$, $p>0,$ $q\ge 1$ and $u$ is a positive classical solution of (\ref{eq 1.1})
  such that $u\in L^p(\R^N)$.
  Then $I_\alpha[u^p]\in L^\infty(\R^N\setminus B_1(0))$,
  \begin{equation}\label{2.1.1}
\lim_{|x|\to+\infty} I_\alpha[u^p](x)=0
  \end{equation}
and  for any $\theta\in(0,1)$,  there holds
\begin{equation}\label{2.01}
\lim_{|x|\to+\infty} u(x)  e^{\theta |x|}=0.
\end{equation}

 \end{lemma}
\begin{proof} For any $x\in \R^N\setminus B_1(0)$, we have that
\begin{eqnarray*}
I_\alpha[u^p](x) &=& \int_{\R^N}\frac{u(y)^p}{|x-y|^{N-\alpha}}dy   \\
    &=& \int_{\R^N\setminus B_{\frac 12}(x)}\frac{u(y)^p}{|x-y|^{N-\alpha}}dy+\int_{  B_{\frac 12}(x)}\frac{u(y)^p}{|x-y|^{N-\alpha}}dy
    \\&\le & \left(\frac 12\right)^{\alpha-N}\norm{u}^p_{L^p(\R^N)} +\norm{u}_{L^\infty(B_{\frac 12}(x))}^p\int_{  B_{\frac 12}(x)}\frac{1}{|x-y|^{N-\alpha}}dy
   \\&\le & \left(\frac 12\right)^{\alpha-N} \norm{u}^p_{L^p(\R^N)} + c_4 \norm{u}^p_{L^\infty(\R^N\setminus B_{\frac12}(0))},
\end{eqnarray*}
thus $I_\alpha[u^p]\in L^\infty(\R^N\setminus B_1(0))$.

Similarly, for $x\in \R^N\setminus B_2(0)$ and $r \in (0,\frac{|x|}{2})$ depending on $|x|$, which will be chosen later,
we have
\begin{eqnarray*}
I_\alpha[u^p](x)
    &\le & r^{\alpha-N}\norm{u}^p_{L^p(\R^N)} +\norm{u}_{L^\infty(B_r(x))}^p\int_{ B_r(x)}\frac{1}{|x-y|^{N-\alpha}}dy
    \\&\le & r^{\alpha-N} \norm{u}^p_{L^p(\R^N)} +r^\alpha \norm{u}^p_{L^\infty(\R^N\setminus B_{|x|-r}(0))}
      \\&\le & r^{\alpha-N} \norm{u}^p_{L^p(\R^N)} +r^\alpha \norm{u}^p_{L^\infty(\R^N\setminus B_{\frac{|x|}2}(0))}.
\end{eqnarray*}
Since
$$\lim_{|x|\to+\infty}\norm{u}_{L^\infty(\R^N\setminus B_{\frac{|x|}2}(0))}=0,$$
then
$r: =\min\{\norm{u}^{-\frac{p}{2\alpha}}_{L^\infty(\R^N\setminus B_{\frac{|x|}2}(0))},\, \frac{|x|}{4}\}\to +\infty$ as $|x|\to+\infty$,
and thus
$$\lim_{|x|\to+\infty}  r^{\alpha-N} \norm{u}^p_{L^p(\R^N)}=0\quad{\rm and}\quad\lim_{|x|\to+\infty} r^\alpha \norm{u}^p_{L^\infty(\R^N\setminus B_{\frac{|x|}2}(0))}=0, $$
which imply that (\ref{2.1.1}) holds.

Now for any $\theta'\in(\theta,1)$, since $q \geq 1$,
there exists $r_1>2$ such that
$$I_\alpha[u^p](x)u(x)^{q-1}\le  1- \theta'\quad {\rm for}\quad |x|\ge r_1$$
and
$$-\Delta   e^{-\theta'|x|} +  \theta' e^{-\theta'|x|}\ge 0, \quad \quad x\in\R^N\setminus B_{r_1}(0).$$
Then we have that
$$
 \arraycolsep=1pt
\begin{array}{lll}
 \displaystyle   -\Delta   u+  \theta' u \le 0\quad
  {\rm in}\quad  \mathbb{R}^N\setminus B_{r_1}(0),\\[2mm]
 \phantom{ -\, \,\, }
 \displaystyle   \lim_{|x|\to+\infty}u(x)=0.
\end{array}
$$
 It follows by Comparison Principle that
 $$u(x)\le c_5e^{-\theta'|x|}\quad{\rm for}\quad   |x|\ge r_1,$$
 which implies that (\ref{2.01}) holds.
\end{proof}

 When $q\ge 1$, exponential  decay of the solutions to equation (\ref{eq 1.1})  enables us to focus on the singularities at
the origin. Precisely, from Lemma \ref{lm 2.2} and Lemma \ref{lm 2.1}, we have the following conclusion.
\begin{proposition}\label{pr 2.4}
Under the assumptions of Lemma \ref{lm 2.1}, we have that
  $$I_\alpha[u^p]u^q\in  L^1(\R^N)\quad{\rm and}\quad  u\in L^1(\R^N).$$
Furthermore, if $u\in L^t (B_1(0))$ for $t\in[1,+\infty]$, then
$u\in L^t(\R^N)$.
\end{proposition}
\begin{proof}
 From Lemma \ref{lm 2.2}, we know that $I_\alpha[u^p]u^q\in  L^1_{loc}(\R^N)$ and by Lemma \ref{lm 2.1},
 we have that $I_\alpha[u^p]u^q\in  L^1(\R^N)$.
Since $u$ is a positive solution, then
$$u\ge c_6\quad{\rm in}\quad B_2(0)\setminus B_{\frac12}(0)$$
and for $x\in B_1(0)\setminus\{0\}$,
\begin{eqnarray}
 I_\alpha[u^p](x) &\ge & \int_{B_2(0)\setminus B_{\frac12}(0)} \frac{u(x)^p}{|x-y|^{N-\alpha}}dy  \label{3.1+1}
    \ge  c_7.
\end{eqnarray}
 Then
 \begin{eqnarray*}
   \int_{B_1(0)} u(x) dx &\le &|B_1(0)|^{1-\frac1q} \left(\int_{B_1(0)} u(x)^q dx\right)^\frac1q \\
     &\le & |B_1(0)|^{1-\frac1q} \left(\frac1{c_7} \int_{B_1(0)}I_\alpha[u^p] u^q dx \right)^\frac1q<+\infty,
 \end{eqnarray*}
 that is, $u\in L^1_{loc}(\R^N)$. By Lemma \ref{lm 2.1}, it implies that $u\in L^1(\R^N)$.

If   $u\in L^t (B_1(0))$ for $t\in[1,+\infty]$, it infers by Lemma \ref{lm 2.1} that $u\in L^t(\R^N)$.
\end{proof}

To tackle the singularity estimate (\ref{1.2}), we establish the following lemma.
\begin{lemma}\label{lm 2.3}
Assume that $\alpha,\tau\in(0,N)$ and  $V_\tau:\R^N\setminus\{0\}\to\R_+$
satisfies
$$ V_\tau(x)\le c_8|x|^{-\tau} \quad{\rm for}\ \ x\in B_2(0).$$
If $V_\tau\in   L^\infty (\R^N\setminus B_1(0))$, then for $x\in B_{\frac12}(0)\setminus\{0\}$,
\begin{equation}\label{4.2.6}
\mathbb{G}[ V_\tau](x)\le
\left\{ \arraycolsep=1pt
\begin{array}{lll}
 c_9|x|^{-\tau+2} \quad
 &{\rm if}\quad \tau>2,\\[2mm]
 \displaystyle   -c_9\log|x| \quad
 &{\rm if}\quad \tau=2,\\[2mm]
 c_9  \quad
 &{\rm if}\quad \tau<2.
\end{array}
\right.
\end{equation}
If $V_\tau\in   L^1(\R^N)$, then for $x\in B_{\frac12}(0)\setminus\{0\}$,
\begin{equation}\label{4.2.6-1}
I_\alpha[ V_\tau](x)\le
\left\{ \arraycolsep=1pt
\begin{array}{lll}
 c_9|x|^{-\tau+\alpha} \quad
 &{\rm if}\quad \tau>\alpha,\\[2mm]
 \displaystyle   -c_9\log|x| \quad
 &{\rm if}\quad \tau=\alpha,\\[2mm]
 c_9  \quad
 &{\rm if}\quad \tau<\alpha.
\end{array}
\right.
\end{equation}

\end{lemma}
\begin{proof} Since the Green's function $G(x,\cdot)$ decays exponentially, then for $x\in B_{\frac12}(0)\setminus\{0\}$, we have that
\begin{eqnarray*}
\mathbb{G}[ V_\tau](x)  &\le& c_{10}   \int_{B_1(0)}\frac1{|x-y|^{N-2}}\frac1{|y|^{\tau}}dy +c_{10}\norm{V_\tau}_{L^\infty(\R^N\setminus B_1(0))}  \\
    &=&   c_{10}  |x|^{2-\tau } \int_{B_{\frac{1}{|x|}}(0)}\frac1{|e_{x}-y|^{N-2}}\frac1{|y|^{\tau}}dy+c_{10}\norm{V_\tau}_{L^\infty(\R^N\setminus B_1(0))}
    \\&\le &c_{11} |x|^{2-\tau }\int_{B_{\frac{1}{|x|}}(0)} \frac{1}{1+|y|^{ N-2+\tau}} dy+c_{10}\norm{V_\tau}_{L^\infty(\R^N\setminus B_1(0))}
   \\&\le & \left\{ \arraycolsep=1pt
\begin{array}{lll}
 c_9|x|^{-\tau+2} \quad
 &{\rm if}\quad \tau>2,\\[2mm]
 \displaystyle   -c_9\log|x| \quad
 &{\rm if}\quad \tau=2,\\[2mm]
 c_9  \quad
 &{\rm if}\quad \tau<2,
\end{array}
\right.
\end{eqnarray*}
where $e_{x}=\frac{x}{|x|}$. This implies (\ref{4.2.6}).

Observing  that the Riesz potential decays polynomially,  it infers that for $x\in B_{\frac12}(0)\setminus\{0\}$,
\begin{eqnarray*}
I_\alpha[ V_\tau](x)  &\le& c_{10}   \int_{B_1(0)}\frac1{|x-y|^{N-\alpha}}\frac1{|y|^{\tau}}dy +c_{10}\norm{V_\tau}_{L^1(\R^N)}  \\
    &=&   c_{10}   |x|^{\alpha-\tau } \int_{B_{\frac{1}{|x|}}(0)}\frac1{|e_{x}-y|^{N-\alpha}}\frac1{|y|^{\tau}}dy+c_{10}\norm{V_\tau}_{L^1(\R^N)}
    \\&\le &c_{12} |x|^{\alpha-\tau }\int_{B_{\frac{1}{|x|}}(0)} \frac{1}{1+|y|^{ N-\alpha+\tau}} dy+c_{10}\norm{V_\tau}_{L^1(\R^N)}
   \\&\le & \left\{ \arraycolsep=1pt
\begin{array}{lll}
 c_9|x|^{-\tau+\alpha} \quad
 &{\rm if}\quad \tau>\alpha,\\[2mm]
 \displaystyle   -c_9\log|x| \quad
 &{\rm if}\quad \tau=\alpha,\\[2mm]
 c_9  \quad
 &{\rm if}\quad \tau<\alpha.
\end{array}
\right.
\end{eqnarray*}
This ends the proof.\end{proof}

Applying Lemma \ref{lm 2.3},
we have the following corollary about
 the estimates for  $\mathbb{G}[(I_\alpha[V_\tau^p])^t]$ and $\mathbb{G}[(V_\tau^q)^t]$.

\begin{corollary}\label{lm 2.4}
Let $\alpha\in(0,N)$, $p\in(0,\frac{N}{N-2}),\,q\in(1,\frac{N}{N-2})$
 and $V_\tau(x)\le c_8|x|^{-\tau}$ for $x\in B_2(0)$ with $\tau\in(0,N-2]$.

If $V_\tau\in   L^{p}(\R^N\setminus B_1(0))\cap L^\infty (\R^N\setminus B_1(0))$ and $(p\tau-\alpha)t<N$, then for $x\in B_{\frac12}(0)\setminus\{0\}$,
\begin{equation}\label{2.1.3}
\mathbb{G}[(I_\alpha[V_\tau^p])^t](x)\le
\left\{ \arraycolsep=1pt
\begin{array}{lll}
 c_{13}|x|^{t(\alpha-p\tau)+2} \quad
 &{\rm if}\quad \tau>\frac1{p}(\alpha+\frac 2t),\\[2mm]
 \displaystyle   - c_{13}\log|x| \quad
 &{\rm if}\quad \tau=\frac1{p}(\alpha+\frac 2t),\\[2mm]
 c_{13}  \quad
 &{\rm if}\quad \tau<\frac1{p}(\alpha+\frac 2t).
\end{array}
\right.
\end{equation}
If $V_\tau\in C(\R^N\setminus\{0\})\cap L^\infty (\R^N\setminus B_1(0)) $ and $\tau qt<N$, then for $x\in B_{\frac12}(0)\setminus\{0\}$,
\begin{equation}\label{2.1.4}
\mathbb{G}[(V_\tau^q)^t](x)\le
\left\{ \arraycolsep=1pt
\begin{array}{lll}
  c_{13}|x|^{-\tau qt+2} \quad
 &{\rm if}\quad \tau>\frac2{qt},\\[2mm]
 \displaystyle   - c_{13}\log|x| \quad
 &{\rm if}\quad \tau=\frac2{qt},\\[2mm]
 c_{13}  \quad
 &{\rm if}\quad \tau<\frac2{qt}.
\end{array}
\right.
\end{equation}
\end{corollary}
 \begin{proof}
   For $y\in \R^N\setminus B_1(0)$, we have that
\begin{eqnarray*}
I_\alpha[ V_\tau^p](y)  &\le& c_8^p  \int_{B_1(0)}\frac1{|y-z|^{N-\alpha}}\frac1{|z|^{p\tau}}dz+c_{10}\norm{V_\tau}^p_{L^p(\R^N\setminus B_1(0))}
\\[2mm]
&\le &c_8^p \norm{V_\tau}^p_{L^\infty(\R^N\setminus B_1(0))}+c_{10}\norm{V_\tau}_{L^1(\R^N\setminus B_1(0))},
\end{eqnarray*}
that is, $I_\alpha[V_\tau^p]\in {L^\infty(\R^N\setminus B_1(0))}$.  Now we apply Lemma \ref{lm 2.3}  to  obtain (\ref{2.1.3}).
It is similar to obtain (\ref{2.1.4}).
 \end{proof}

\medskip

\begin{lemma}\label{lm2.6}
Assume that  $\alpha\in(0,N)$,
 $p>0$,  $ q\ge 1$ and   $u\in L^1(\R^N)$ is a positive weak solution of (\ref{eq 1.2}) with $k\ge 0$.
  Then
$$
 u\ge k\Gamma_0\quad {\rm a.e.\ in}\quad \R^N,
$$
 where $\Gamma_0$ is the fundamental solution of $-\Delta+id$.

\end{lemma}
\begin{proof}
Let $w=u-k\Gamma_0$, then $w$ is a weak solution of
$$
\arraycolsep=1pt
\begin{array}{lll}
 \displaystyle\ \, -\Delta w+w=I_\alpha[u^p]u^q\quad
 &{\rm in}\quad \R^N,\\[2mm]
 \phantom{  }
 \displaystyle   \lim_{|x|\to+\infty}w(x)=0,
\end{array}
$$
in the distribution sense, that is
$$\int_{\R^N}w(-\Delta \xi+ \xi) dx=\int_{\R^N}I_\alpha[u^p]u^q \xi dx,\qquad\forall\, \xi\in C^{\infty}_c(\R^N).$$

For any  $n\in \N$, denote
$$\xi_{n}(x)=\mathbb{G}[{\rm sign}(w_-)](x)\eta_0(\frac xn),\qquad\forall\, x\in\R^N,$$
where $t_-=\min\{t,0\}$ and $\eta_0:\R^N\to [0,1]$ is a $C^\infty$-function with the support in $B_2(0)$ and
 satisfying $\eta_0=1$ in $B_1(0)$, then  $\xi_{n}\in C^{\infty}_c(\R^N)$ for any $n\in\N$. Thus,
 we have that
 \begin{eqnarray*}
\int_{\R^N}&w&(-\Delta \xi_n+\xi_n)  dx = \int_{\R^N}w(x)  {\rm sign}(w_-)(x) \eta_0(\frac xn)\, dx
+\\&&\int_{\R^N} w(x) \nabla \mathbb{G}[{\rm sign}(w_-)](x)\cdot \nabla \eta_0(\frac xn)\,dx+
\int_{\R^N} w(x)   \mathbb{G}[{\rm sign}(w_-)](x) (-\Delta)\eta_0(\frac xn)\,dx,
 \end{eqnarray*}
 where
 $$
  \int_{\R^N}w(x)  {\rm sign}(w_-)(x) \eta_0(\frac xn) \,dx=-\int_{\R^N}w_-(x)\eta_0 (\frac xn)\,dx\ge \int_{B_n(0)}|w_-(x)|\,dx,
$$
 \begin{eqnarray*}
 \left|\int_{\R^N} w(x) \nabla \mathbb{G}[{\rm sign}(w_-)](x)\cdot \nabla \eta_0(\frac xn)\,dx\right|  &\le & \frac {c_{14}}n\int_{B_{2n}(0)\setminus B_n(0)} |w(x)|\, dx,
 \end{eqnarray*}
 and
\begin{eqnarray*}
 \left| \int_{\R^N} w (x)  \mathbb{G}[{\rm sign}(w_-)] (x)(-\Delta)\eta_0(\frac xn)\,dx \right| &\le &
 \frac {c_{14}}{n^2}\int_{B_{2n}(0)\setminus B_n(0)} |w(x)|\, dx.
 \end{eqnarray*}
Since $I_\alpha[u^p]u^q\in L^1(\R^N)$ and $\xi_{n}$ is non-positive in $\R^N$, we have
\begin{eqnarray*}
\int_{\R^N}I_\alpha[u^p]u^q \xi_{n} dx  \le   0,
\end{eqnarray*}
which implies that
\begin{eqnarray*}
\int_{B_{n}(0)}|w_-(x)|dx  &\le &   \left| \int_{\R^N} w (x)  \mathbb{G}[{\rm sign}(w_-)] (x)(-\nabla)\eta_0(\frac xn)\,dx \right|
\\&&+\left| \int_{\R^N} w(x)   \mathbb{G}[{\rm sign}(w_-)](x) (-\Delta)\eta_0(\frac xn)\,dx \right|.
\end{eqnarray*}
Therefore by taking $n\to\infty$, we obtain that
 $$w_-=0\quad{\rm a.e.\ in}\ \ \R^N,$$
that is,
 $u -k\Gamma_0\ge0$ a.e. in $\R^N$.
\end{proof}
\vskip0.5cm

\section{Classification of singularities}

In this section, we classify the singularities of positive solutions to equation (\ref{eq 1.1}).

\begin{proposition}\label{pr 2.1}
Assume that $N\ge 3$, $\alpha\in(0,N)$,
 $p>0$,  $ q\ge 1$
 and $u$ is a positive classical solution of (\ref{eq 1.1}) satisfying $u\in L^p(\R^N)$.
 Then $u$ is a weak solution of (\ref{eq 1.2}) for some $k\ge0$.
Furthermore, if $(p,\,q)$ satisfies (\ref{1.3}),
then $k=0$.
\end{proposition}
\begin{proof}
From Proposition \ref{pr 2.4}, we know that $I_\alpha[u^p]u^q\in L^1(\R^N)$ and $u\in L^1(\R^N)$. Define the operator $L$  by the following
\begin{equation}\label{3.1}
L(\xi):=\int_{\R^N} [u(-\Delta \xi+\xi) -I_\alpha[u^p]u^q\xi]\,dx,\quad \forall\xi\in C^\infty_c(\R^N).
\end{equation}
First we claim that for any $\xi\in C^\infty_c(\R^N)$ with
the support in $\R^N\setminus\{0\}$,
$$L(\xi)=0.$$
In fact, since $\xi\in C^\infty_c(\R^N)$ has the support in $\R^N\setminus\{0\}$, then there exists $r\in(0,1)$ such that
$\xi=0$ in $B_r(0)\cup (\R^N\setminus B_{\frac1r}(0))$ and  then
\begin{eqnarray*}
L(\xi)
&=&  \int_{B_{\frac1r}(0)\setminus B_r(0)}[ u( -\Delta  \xi+\xi)  -I_\alpha[u^p]u^q\xi]\,dx
\\&  = &  \int_{B_{\frac1r}(0)\setminus B_r(0)}(-\Delta u+u  -I_\alpha[u^p]u^q )\xi\,dx
\\&=&0.
\end{eqnarray*}
From Theorem 1.1 in \cite{L},    it implies that
\begin{equation}\label{S}
 L=  k \delta_0\quad {\rm for\ some\ }\ k\ge0,
\end{equation}
 that is,
\begin{equation}\label{3.3}
L(\xi)= \int_{\R^N} \left[u(-\Delta \xi+\xi) -I_\alpha[u^p]u^q\xi\right]\,dx=  k \xi(0),\quad  \quad \forall \xi\in C^\infty_c (\R^N).
\end{equation}
Then $u$ is a weak solution of (\ref{eq 1.2}) for some $k\ge0$.

 Next we prove that  $k=0$ if $(p,q)$ satisfies (\ref{1.3}).  By contradiction, if $k>0$, then Lemma \ref{lm2.6} implies that
 $$u\ge k\Gamma_0\quad {\rm in}\quad B_1(0)\setminus\{0\}.$$
It is well known that
 $$\Gamma_0(x)\ge c_{15}|x|^{2-N}, \qquad \forall x\in B_1(0)\setminus\{0\}.$$

{\it Case I: $p+q\ge \frac{N+\alpha}{N-2}$.}
For $x\in B_1(0)\setminus\{0\}$, we have that
 \begin{eqnarray*}
I_\alpha[u^p]u^q (x)&\ge & k^{p+q} I_\alpha[\Gamma_0^p]\Gamma_0^q(x)\\
    &>&  c_{15} k^{p+q} \bigg(\int_{B_1(0)}\frac{|y|^{(2-N)p}}{|x-y|^{N-\alpha}}\, dy\bigg)\, |x|^{(2-N)q}
    \\&=&c_{15} k^{p+q} \bigg(\int_{B_{\frac{1}{|x|}}(0)}\frac{|z|^{(2-N)p}}{|e_x-z|^{N-\alpha}}\, dz\bigg)\, |x|^{(2-N)(p+q)+\alpha}
   \\&\ge&c_{15} k^{p+q} \bigg(\int_{B_1(0)}\frac{|z|^{(2-N)p}}{|e_x-z|^{N-\alpha}}\, dz\bigg)\, |x|^{(2-N)(p+q)+\alpha},
 \end{eqnarray*}
 where $e_x=\frac{x}{|x|}$.
 But in {\it Case} $I$,
 the function $|\cdot|^{(2-N)(p+q)+\alpha}$ does not belong to $L^1_{loc}(\R^N)$.
This contradicts Lemma \ref{lm 2.2} and we have that $k=0$.

 \smallskip

 {\it Case II: $p\ge \frac{N}{N-2}$.} We note that
 $$\Gamma_0^p(x)\ge c_{15}^p|x|^{p(2-N)}\ge c_{15}^p |x|^{-N}\quad{\rm for}\quad 0<|x|<1,$$
 then $I_\alpha[\Gamma_0^p]\equiv +\infty$ in $B_1(0)\setminus\{0\}$ and then for $x\in B_1(0)\setminus\{0\}$,
we have that
 \begin{eqnarray*}
I_\alpha[u^p](x) \ge   k^{p+q} I_\alpha[\Gamma_0^p](x)=+\infty,
 \end{eqnarray*}
 which is impossible. Thus $k=0$.

  \smallskip

 {\it Case III: $q\ge \frac{N}{N-2}$.}  We note that $\Gamma_0^q(x)\ge c_{15}^q|x|^{q(2-N)}$ for $0<|x|<1$,
 where $q(2-N)\le -N$. It follows from (\ref{3.1+1}) that $I_\alpha[u^p]\ge c_7$ in $ B_1(0)\setminus\{0\}$, then
$$I_\alpha[u^p]u^q (x) \ge   c_{7}k^{q}\Gamma_0^q(x)\ge  c_{7}k^{q} c_{15}^q|x|^{q(2-N)}\ge  c_{7} c_{15}^qk^{q}|x|^{-N}\quad{\rm for}\ \ 0<|x|<1,$$
  which contradicts Lemma \ref{lm 2.2}. Therefore we have that $k=0$.
\end{proof}

Now we focus on the subcritical case.
\begin{proposition}\label{pr 2.2}
Assume that $N\ge 3$, $\alpha\in(0,N)$, (\ref{1.4}) holds for $p>0$, $ q\ge 1$ and
  $u$ is a positive classical solution of (\ref{eq 1.1}) satisfying $u\in L^p(\R^N)$. Assume more that $u$ is   a   weak solution of (\ref{eq 1.2}) with $k=0$.
Then $u$ is a classical solution of (\ref{eq 1.3}).

\end{proposition}
\begin{proof}
Since $I_\alpha[u^p]u^q\in L^1(\R^N)$ and $k=0$, we have that
  $$-\Delta u+u=  I_\alpha[u^p]u^q\quad {\rm in\ the\ distribution\ sense }  $$
and then  $u\in L^{t}(\R^N)$ with $t<\frac{N}{N-2}$.

\smallskip

{\it Case 1: $p<\frac{\alpha}{N-2}$.} it follows from  Proposition \ref{embedding1} that  $I_\alpha[u^p]\in L^{\infty}_{loc}(\R^N)$.
Then applying the standard bootstrap argument, we have that $u\in L^\infty(\R^N)$ and then $u$ is a classical solution of (\ref{eq 1.3}).

\smallskip

{\it Case 2: $p=\frac{\alpha}{N-2}$.} Again by Proposition \ref{embedding1}, we see that $I_\alpha[u^p]\in L^{t}_{loc}(\R^N)$ for any $t>1$.
By H\"{o}lder's inequality, we have that
\begin{eqnarray}\label{2.5}
\int_{B_1(0)} (I_\alpha[u^p]u^q)^s dx&\le &\left[\int_{B_1(0)} (I_\alpha[u^p]) ^{st} dx\right]^{\frac1t} \left[\int_{B_1(0)} u ^{qs\frac t{t-1}} dx\right]^{1-\frac1t}
\end{eqnarray}
for $s,t>1$  satisfying
$$
\left\{  \arraycolsep=1pt
\begin{array}{lll}
 \displaystyle   st<+\infty,\\[2mm]
 \displaystyle   s\frac t{t-1}< \frac1q\frac{N}{N-2}.
\end{array}
\right.
$$
Since $q<\frac{N}{N-2}$, we choose $t$ big enough, then
$$I_\alpha[u^p]u^q\in L^s(\R^N)$$
for any $s\in (1,\frac1q\frac{N}{N-2})$. Then by Proposition \ref{embedding}, $u\in L^{\frac{Ns}{N-2s}}(\R^N)$ and $\frac{1}{p}\frac{Ns}{N-2s-\alpha s}>\frac{N}{\alpha}$,
thus $I_\alpha[u^p]\in L^\infty(\R^N)$ and by standard elliptic regularity theory, $u$ is a classical solution of (\ref{eq 1.3}).

\smallskip

{\it Case 3: $p>\frac{\alpha}{N-2}$.} We have that $I_\alpha[u^p]\in L^{\frac{N\theta}{N-\alpha\theta}}_{loc}(\R^N)$  for any $\theta<\frac1p\frac{N}{N-2}$.
By H\"{o}lder's inequality, (\ref{2.5}) holds
for $s \geq 1, t>1$  satisfying
\begin{equation}\label{2.2}
\left\{  \arraycolsep=1pt
\begin{array}{lll}
 \displaystyle  st<\frac{N\frac1p\frac{N}{N-2}}{N-\alpha\frac1p\frac{N}{N-2}}=\frac{N}{p(N-2)-\alpha},\\[4mm]
 \displaystyle  s\frac t{t-1}< \frac1q\frac{N}{N-2}.
\end{array}
\right.
\end{equation}
When  $s=1$,  (\ref{2.2}) reduces to
\begin{equation}\label{2.3}
   \frac{N}{N-q(N-2)}<  t<\frac{N}{p(N-2)-\alpha}.
\end{equation}
Clearly the existence of $t$ salifying (\ref{2.3}) is guaranteed by (\ref{1.4}).
Now choose
\begin{equation}\label{2.3+1}
t=t_1:= \frac{(p+q)(N-2)-\alpha}{p(N-2)-\alpha}
\end{equation}
 such that
 \begin{equation}\label{2.3+2}
\frac1{t_1}\frac{N}{p(N-2)-\alpha}=\frac{t_1-1}{t_1}\frac1q\frac{N}{N-2},
\end{equation}
holds, then (\ref{2.2}) becomes to
$$
   s<\frac{N}{(p+q)(N-2)-\alpha}
$$
and
$$I_\alpha[u^p]u^q\in L^{s}(\R^N)\quad {\rm for\ any }\ s< \frac{N}{(p+q)(N-2)-\alpha}.$$

If $\frac{N}{(p+q)(N-2)-\alpha}> \frac N2$,  by Proposition \ref{embedding}, it implies that $u\in L^\infty(\R^N)$,  then  $u$ is a classical solution of (\ref{eq 1.3}).

If $\frac{N}{(p+q)(N-2)-\alpha}\le \frac N2$,  fix some $s_1\in(1,\frac{N}{(p+q)(N-2)-\alpha}),$
then
 $u^p\in L^{\theta}(\R^N)$ with $\theta\le\frac1p\frac{Ns_1}{N-2s_1}$ and it follows by Proposition \ref{embedding1} that
$I_\alpha[u^p]\in L^{\frac{N\theta}{N-\alpha\theta}}_{loc}(\R^N)$  for any $\theta\le\frac1p\frac{Ns_1}{N-2s_1}$.
Now (\ref{2.5}) holds  for $s,t>1$  satisfying
\begin{equation}\label{2.2.1}
\left\{  \arraycolsep=1pt
\begin{array}{lll}
 \displaystyle   st\le  \frac{Ns_1}{p(N-2s_1)-\alpha s_1},\\[4mm]
 \displaystyle   s\frac t{t-1}\le \frac1q\frac{Ns_1}{N-2s_1}.
\end{array}
\right.
\end{equation}
Take  $s=s_1$, then (\ref{2.2.1}) reduces to
\begin{equation}\label{2.3.1}
   \frac{N}{N-q(N-2s_1)}\le t\le\frac{N}{p(N-2s_1)-\alpha s_1}.
\end{equation}
 Choose $t=t_2:=\frac{(p+q)(N-2s_1)-\alpha s_1}{p(N-2s_1)-\alpha s_1}$ and then
$$\frac1{t_2}\frac{Ns_1}{p(N-2s_1)-\alpha s_1}=\frac{t_2-1}{qt_2} \frac{Ns_1}{N-2s_1}.$$
Condition (\ref{2.2.1}) becomes to
\begin{equation}\label{2.4}
   s\le \frac1{t_2}\frac{Ns_1}{p(N-2s_1)-\alpha s_1}.
\end{equation}
Choose $s=s_2:=\frac1{t_2}\frac{Ns_1}{p(N-2s_1)-\alpha s_1}=  \frac{Ns_1}{(p+q)(N-2s_1)-\alpha s_1}$, then  $I_\alpha[u^p]u^q\in L^{s_2}(\R^N)$ and
$$s_2> \frac{N}{(p+q)(N-2)-\alpha}s_1.$$

If $s_2>\frac{N}2$, we are done. If not, step by step, assume that $u\in L^{s_{n-1}}(\R^N)$ with $s_{n-1}<\frac{N}2$, then we can find $s>s_{n-1}$ such that
$I_\alpha[u^p]u^q\in L^{s}(\R^N)$  and
(\ref{2.5}) holds  for $s,t>1$  satisfying
\begin{equation}\label{2.2.1-}
\left\{  \arraycolsep=1pt
\begin{array}{lll}
 \displaystyle   st\le  \frac{Ns_{n-1}}{p(N-2s_{n-1})-\alpha s_{n-1}},\\[4mm]
 \displaystyle   s\frac t{t-1}\le \frac1q\frac{Ns_{n-1}}{N-2s_{n-1}}.
\end{array}
\right.
\end{equation}
Choose $ t_n:=\frac{(p+q)(N-2s_{n-1})-\alpha s_{n-1}}{p(N-2s_{n-1})-\alpha s_{n-1}}$
and
$$ s_n  =  \frac{N}{(p+q)(N-2s_{n-1})-\alpha s_{n-1}}s_{n-1}.$$
Observing that $s_n>1$ and $\{s_n\}_n$ is increasing with respect to $n$  satisfying
\begin{eqnarray*}
    s_n    &\ge &   \frac{N}{(p+q)(N-2 )-\alpha }  s_{n-1}\\
    &\ge &  \left( \frac{N}{(p+q)(N-2 )-\alpha } \right)^{n-1} s_1
    \to +\infty\quad{\rm as}\quad n\to+\infty.
\end{eqnarray*}
Then there exists $n_0$ such that $s_{n_0 -1} \le \frac N2$ and  $s_{n_0}>\frac N2$, thus we have that $\mathbb{G}[I_\alpha[u^p]u^q]\in L^\infty(\R^N)$
and the rest of the proof is standard  to obtain that $u$ is a classical solution of (\ref{eq 1.3}).
\end{proof}

Next we consider the subcritical case with $k>0$. We have the following

\begin{proposition}\label{pr 2.3}
Assume that $\alpha\in(0,N)$, (\ref{1.4}) holds for $p>0$, $ q\ge 1$ and $u$ is a positive classical solution of (\ref{eq 1.1})  satisfying $u\in L^p(\R^N)$.
Assume more that $u$ is a  weak solution of (\ref{eq 1.2}) with $k>0$.

Then \begin{equation}\label{2.3.2}
 \lim_{|x|\to0^+} u(x)|x|^{N-2}=c_{N} k.
\end{equation}
\end{proposition}
\begin{proof}
Observe that
$$ \lim_{|x|\to0^+} \Gamma_0(x)|x|^{N-2} =c_{N}$$
and
\begin{equation}\label{13.2.1a}
u=\mathbb{G}[I_\alpha[u^p]u^q]+k\Gamma_0,
\end{equation}
then   $u^p\in L^{t}(\R^N)$ with $t< \frac1p\frac{N}{N-2}$.

{\it Case 1: $p<\frac{\alpha}{N-2}$. } We see that $\frac1p\frac{N}{N-2}>\frac{N}{\alpha}$,  then it follows from Proposition \ref{embedding1} that
$$I_\alpha[u^p]\in L^\infty(\R^N).$$
In this case,
(\ref{13.2.1a}) could be reduced to
\begin{equation}\label{3.0.2}
u\le c_{16}\mathbb{G}[u^q]+k\Gamma_0,
\end{equation}
then it follows by \cite[Theorem 1.1]{L} that (\ref{2.3.2}) holds.

\smallskip

{\it Case 2: $p= \frac{\alpha}{N-2}$.}  We observe that $I_\alpha[u^p]\in L^t(\R^N)\quad{\rm for\ any}\ t>1.$
For any $s<\frac1q \frac{N}{N-2}$, there exists $\bar t>1$ such that
$$s\frac {\bar t}{\bar t-1}<\frac1q \frac{N}{N-2}$$
holds.
Then by using again (\ref{2.5}) with $t=\bar t$, we get $I_\alpha[u^p]u^q\in L^s(\R^N)$ for any $s<\frac1q \frac{N}{N-2}$.
Let $u_1: =\mathbb{G}[I_\alpha[u^p]u^q]$, then $u= u_1+k\Gamma_0$.  By
Young's inequality and the fact that $(a+b)^r\le 2^r(a^r+b^r)$ for $a,b,r>0$, we have that
\begin{eqnarray*}
   \displaystyle  u_1 &=& \mathbb{G}[I_\alpha[(u_1+k\Gamma_0)^p](u_1+k\Gamma_0)^q ] \\[1.5mm]
   \displaystyle  &\le & c_{17} \left(\mathbb{G}[I_\alpha[u_1^p]u_1^q] +k^p\mathbb{G}[I_\alpha[\Gamma_0^p]u_1^q]+k^q \mathbb{G}[I_\alpha[u_1^p]\Gamma_0^q] +\mathbb{G}[I_\alpha[ \Gamma_0^p]\Gamma_0^q]\right)
   \\
   \displaystyle  &\le &  c_{18}  \mathbb{G}\left[I_\alpha[u_1^{p}]^{\bar t}\right]  + c_{18} \mathbb{G}\left[  u_1^{\frac{q\bar t }{\bar t -1}}\right]
  +\Gamma_1,
\end{eqnarray*}
where
$$\Gamma_1=c_{18} \mathbb{G}\left[  I_\alpha[\Gamma_0^p]^{\bar t } +   \Gamma_0^{\frac{q\bar t }{\bar t -1}}\right]\le  c_{19} +c_{18}\mathbb{G}\left[\Gamma_0^{\frac{q\bar t }{\bar t -1}}\right]$$
by the fact that $I_\alpha[\Gamma_0^p](x)\le -c_9\log |x| $ for $0<|x|<\frac12$  and $\mathbb{G}\left[  I_\alpha[\Gamma_0^p]^{\bar t }\right]\in L^\infty(\R^N)$.
Since $I_\alpha[u_1^{p}]^{\bar t}\in L^\theta(\R^N)$ for any $\theta>1$, we obtain that $\mathbb{G}\left[ I_\alpha[u_1^{p}]^{\bar t}\right]\in L^\infty(\R^N)$.
Therefore (\ref{13.2.1a}) deduces into
\begin{equation}
  u\le c_{18}\mathbb{G}[u_1^{q\frac{\bar t}{\bar t-1}}]+c_{18}\norm{\mathbb{G}\left[ I_\alpha[u_1^{p}]^{\bar t}\right]}_{L^\infty(\R^N)}+\Gamma_1+k\Gamma_0\quad {\rm in}\quad B_1(0)\setminus\{0\}.
\end{equation}
Then we repeat the procedure in {\it Case 1} since $q\frac{\bar t}{\bar t-1}<\frac{N}{N-2}$.

\smallskip

{\it Case 3: $p> \frac{\alpha}{N-2}$.} We take again $t_1 > 1$ given by (\ref{2.3+1}) such that (\ref{2.3+2}) holds. Since
$ I_\alpha[u^p] \in L^{st_1}(\R^N)$ and $u^q\in L^{s\frac {t_1}{t_1-1}}(\R^N)$ for
  $s<\frac1{t_1}\frac{N}{p(N-2)-\alpha}$,
we obtain that:

if $\frac N{(p+q)(N-2)-\alpha}>\frac{N}{2}$, we have $u_1\in L^\infty(\R^N)$ and we are done;

if not,
re-denote $u_1=\mathbb{G}[I_\alpha[u^p]u^q]\in L^{\frac{N\theta}{N-2\theta}}(\R^N)$ for $\theta\in (1,\frac{N}{(p+q)(N-2)-\alpha})$
if $(p+q)(N-2)-2-\alpha>0$, or for $\theta\in (1,\infty)$
if $(p+q)(N-2)-2-\alpha=0$.
 By   Young's inequality, we have that
\begin{eqnarray*}
   \displaystyle  u_1 &=& \mathbb{G}[I_\alpha[(u_1+k\Gamma_0)^p](u_1+k\Gamma_0)^q ] \\[1.5mm]
   \displaystyle  &\le & c_{19} \left(\mathbb{G}[I_\alpha[u_1^p]u_1^q] +k^p\mathbb{G}[I_\alpha[\Gamma_0^p]u_1^q]+k^q \mathbb{G}[I_\alpha[u_1^p]\Gamma_0^q] +\mathbb{G}[I_\alpha[ \Gamma_0^p]\Gamma_0^q]\right)
   \\
   \displaystyle  &\le &  c_{20}   \mathbb{G}\left[   I_\alpha[u_1^{p}]^{t_1}  +  u_1^{\frac{qt_1 }{t_1 -1}}\right]
  +\Gamma_1,
\end{eqnarray*}
where
$$\Gamma_1=c_{20} \mathbb{G}\left[  I_\alpha[\Gamma_0^p]^{t_1} +   \Gamma_0^{\frac{qt_1 }{t_1 -1}}\right].$$
Let $T_0:= 2-N < 0$.
We notice that if $(p+q)(N-2)-2-\alpha>0$, we have
\begin{eqnarray*}
 \frac{N\frac{N}{(p+q)(N-2)-\alpha}}{N-2\frac{N}{(p+q)(N-2)-\alpha}}  =  \frac{N}{(p+q)(N-2)-\alpha-2}
         >\frac{N}{N-2},
\end{eqnarray*}
and
$$\Gamma_1(x)\le c_{21}|x|^{T_1}\quad {\rm for} \quad 0<|x|<1,$$
where $$T_1: =2+\frac{qt }{t_1 -1}T_0=2+\alpha-(p+q)(N-2)>T_0$$
and thus $u_1\in L^{\frac{N}{N-2}}(\R^N)$.

If $(p+q)(N-2)-2-\alpha=0$,
$$\Gamma_1(x)\le -c_{21}\ln |x| \quad {\rm for} \quad 0<|x|<\frac12,$$
and it is obvious that   $u_1\in L^{\frac{N}{N-2}}(\R^N)$ and
$$I_\alpha[u_1^{p}]^{t_1 }\in L_{loc}^{\frac1{t_1} \frac{N }{p(N-2)-\alpha }}(\R^N)\quad{\rm and} \quad u_1^{\frac{qt_1 }{t_1 -1}}\in L^{\frac{t_1 -1}{qt_1 }\frac{N }{N-2}}(\R^N).$$
Letting
$$s_1=\frac{N}{(p+q)(N-2)-\alpha}\quad{\rm and}\quad  u_2=c_{20}  \mathbb{G}\left[ I_\alpha[u_1^{p}]^{t_1}  +  u_1^{\frac{qt_1}{t_1-1}}\right],$$
we have that $\frac N{N-2}<s_1<\frac{N}{2}$ and $u_1\le u_2+\Gamma_1,$
where
 $u_2\in L^{\frac{Ns_1}{N-2s_1}}(\R^N)$.

  By   Young's inequality, we have that
\begin{eqnarray*}
   \displaystyle  u_2
   \displaystyle  &\le &  c_{22}  \mathbb{G}\left[   I_\alpha[u_2^{p}]^{t_1}  +  u_2^{\frac{qt_1 }{t_1 -1}}\right]
  +\Gamma_2,
\end{eqnarray*}
where
$$\Gamma_2=c_{22} \mathbb{G}\left[  I_\alpha[\Gamma_1^p]^{t_1} +   \Gamma_1^{\frac{qt_1 }{t_1 -1}}\right].$$
We notice that
$$2+(\alpha+pT_1) t_1> 2+\frac{qt_1 }{t_1 -1}T_1>T_1$$
and
$$\Gamma_2(x)\le c_{23}|x|^{T_2}\quad {\rm for} \quad 0<|x|<1,$$
where
$$T_2:=2+\frac{qt_1 }{t_1 -1}T_1>T_1.$$

Note that
$$I_\alpha[u_2^{p}]^{t_1} \in L^{\frac1{t_1}\frac{Ns_1}{p(N-2s_1)-\alpha s_1}}_{loc}(\R^N)\quad\; {\rm and}\quad\;  u_2^{\frac{qt_1 }{t_1 -1}}\in L^{\frac{t_1 -1}{qt_1 }s_1}(\R^N),$$
where
\begin{equation}\label{2.3+3}
\frac1{t_1}\frac{Ns_1}{p(N-2s_1)-\alpha s_1}> \frac1{t_1}\frac{N }{p(N-2 )-\alpha}, \quad\;
 \frac{t_1 -1}{qt_1 }s_1> \frac1{t_1}\frac{N }{p(N-2 )-\alpha}.
 \end{equation}
Then we have that
$$u_2 \in L^{\frac{N\theta}{N-2\theta}}(\R^N)\quad{\rm for}\ \ \theta\in \left(1,\frac{N}{(p+q)(N-2)-\alpha}\right].$$

Inductively, we assume that
$$
u_{n-1}\le c_{n-1}\mathbb{G}[I_\alpha[u_{n-1}^p]^{t_1}+u_{n-1}^{\frac{qt_1}{t_1-1}}]+\Gamma_{n-1}
$$
for some suitable constant $c_{n-1}$. Denote
$$u_n:=c_{n-1} \mathbb{G}\left[I_\alpha[u_{n-1}^{p}]^{t_1}  + u_{n-1}^{\frac{qt_1}{t_1-1}} \right],$$
then $u_{n-1}\le u_n+\Gamma_{n-1}$ and
\begin{equation}\label{2.1-1}
u_n\le c_{n}  \mathbb{G}\left[I_\alpha[u_n^{p}]^{t_1}  +  u_n^{\frac{qt_1 }{t_1 -1}}\right]+ \Gamma_n,
\end{equation}
where
$$\Gamma_n=c_{24} \mathbb{G}\left[  I_\alpha[\Gamma_{n-1}^p]^{t_1} +   \Gamma_{n-1}^{\frac{qt_1 }{t_1 -1}}\right].$$

We notice that
$$I_\alpha[u_{n-1}^{p}]^{t_1} \in L^{\frac1{t_1}\frac{Ns_1}{p(N-2s_1)-\alpha s_1}}_{loc}(\R^N)\qquad{\rm and}\qquad u_{n-1}^{\frac{qt_1 }{t_1 -1}}\in L^{\frac{t_1 -1}{qt_1 }s_1}(\R^N),$$
where $t_1,s_1$ satisfy (\ref{2.3+3}).
Then we get again
$$u_n \in L^{\frac{N\theta}{N-2\theta}}(\R^N)\quad{\rm for}\ \ \theta\in \left(1,\frac{N}{(p+q)(N-2)-\alpha}\right].$$
Furthermore, we have that for $0<|x|<\frac12$,
 $$
\Gamma_n(x)\le
 \arraycolsep=1pt\left\{
\begin{array}{lll}
 \displaystyle  c_{25}|x|^{T_n}\quad
& {\rm if}\quad  T_n<0,\\[2mm]
 \displaystyle   -c_{25}\ln |x|  \quad &{\rm if } \quad T_n=0,
 \\[2mm]
  c_{25} \quad
  &{\rm if}\quad  T_n>0,
\end{array}
\right.
$$
where
$$T_n: =2+\frac{qt_1 }{t_1 -1}T_{n-1}.$$
Since $\frac{qt_1 }{t_1 -1}>1$ and $T_1-T_0>0$, then
\begin{eqnarray*}
 T_n-T_{n-1} &=& \frac{qt_1 }{t_1 -1}( T_{n-1}-T_{n-2})
          \\ &=& \left(\frac{qt_1 }{t_1 -1}\right)^{n-1}(T_1-T_0) \to +\infty\quad{as}\; n\to\infty.
\end{eqnarray*}
Then there exists $n_0\ge 1$ such that
$$T_{n_0}>0\quad {\rm and}\ \ T_{n_0-1}\le 0.$$
Thus, $\Gamma_{n_0}\in L^\infty(\R^N)$ and
$$u\le u_{n_0}+\sum_{i=1}^{n_0-1}\Gamma_i+k\Gamma_0.$$

Finally, our aim is to prove $u_{n_0}\in L^\infty(\R^N)$. Observing that (\ref{2.1-1}) holds for $n=n_0$ and $\Gamma_{n_0}\in L^\infty(\R^N)$, that is,
\begin{eqnarray*}
   \displaystyle  u_{n_0}  \le  c_{26} \left(   \mathbb{G}\left[I_\alpha[u_{n_0}^p]^{t_1}+u_{n_0}^{\frac{q t_1}{t_1-1}}\right]+1\right),
\end{eqnarray*}
where $u_{n_0}\in L^{\frac{Ns_1}{N-2s_1}}(\R^N)$. Then
$$I_\alpha[u_{n_0}^{p}]^{t_1 }\in L_{loc}^{\frac1{t_1} \frac{Ns_1 }{p(N-2s_1)-\alpha s_1 }}(\R^N)\quad{\rm and} \quad u_{n_0}^{\frac{qt_1 }{t_1 -1}}\in L^{\frac{t_1 -1}{qt_1 }\frac{N s_1}{N-2s_1}}(\R^N).$$
We see that, by the definition of $t_1, s_1$,
\begin{eqnarray*}
&&\frac1{t_1} \frac{Ns_1 }{p(N-2s_1)-\alpha s_1 }- \frac{t_1 -1}{qt_1 }\frac{N s_1}{N-2s_1}
\\&&= \frac1{t_1} \frac{N s_1}{p(N-2 )-\alpha }\left[ \frac{\alpha N(s_1-1)}{[p(N-2s_1)-\alpha s_1](N-2s_1)}\right] >0
\end{eqnarray*}
and
\begin{eqnarray*}
  s_2 &:=& \frac{t_1 -1}{qt_1 }\frac{N s_1}{N-2s_1}
   = \frac{N-2}{(p+q)(N-2)-\alpha}\frac{N}{N-2s_1} s_1>s_1
\end{eqnarray*}
by the fact that
 $p+q<\frac{N+\alpha}{N-2}$. Therefore by Proposition \ref{embedding} we obtain that
$$u_{n_0}\in L^{\frac{Ns_2}{N-2s_2}}(\R^N).$$

Inductively, assume that
$$u_{n_0}\in L^{\frac{Ns_{n-1}}{N-2s_{n-1}}}(\R^N)$$
 for $s_{n-1}\in(1,\, \frac N2)$.
then we have that
$$
I_\alpha[u_{n_0}^{p}]^{t_1 }\in
 \arraycolsep=1pt\left\{
\begin{array}{lll}
 \displaystyle   L_{loc}^{\frac1{t_1} \frac{Ns_{n-1} }{p(N-2s_{n-1})-\alpha s_{n-1} }}(\R^N) \quad
 {\rm if}\quad  p(N-2s_{n-1})-\alpha s_{n-1}>0,\\[2mm]
 \displaystyle   L_{loc}^{t}(\R^N)  \quad {\rm for\ any}\quad t>1 \quad{\rm if}\quad p(N-2s_{n-1})-\alpha s_{n-1}=0,
 \\[2mm]
  L_{loc}^{\infty}(\R^N) \quad
  {\rm if}\quad  p(N-2s_{n-1})-\alpha s_{n-1}<0
\end{array}
\right.
$$
and
$$u_{n_0}^{\frac{qt_1 }{t_1 -1}}\in L^{\frac{t_1 -1}{qt_1 }\frac{N s_{n-1}}{N-2s_{n-1}}}(\R^N).$$
For $p(N-2s_{n-1})-\alpha s_{n-1}>0$, we see that
\begin{eqnarray*}
&&\frac1{t_1} \frac{Ns_{n-1} }{p(N-2s_{n-1})-\alpha s_{n-1} }- \frac{t_1 -1}{qt_1 }\frac{N s_{n-1}}{N-2s_{n-1}}
\\&&= \frac1{t_1} \frac{N s_{n-1}}{p(N-2 )-\alpha }\left[ \frac{\alpha N(s_{n-1}-1)}{[p(N-2s_{n-1})-\alpha s_{n-1}](N-2s_{n-1})}\right] >0
\end{eqnarray*}
and
\begin{eqnarray*}
  s_n &:=& \frac{t_1 -1}{qt_1 }\frac{N s_{n-1}}{N-2s_{n-1}}
   =\frac{N-2}{(p+q)(N-2)-\alpha}\frac{N}{N-2s_{n-1}} s_{n-1}>s_{n-1}
\end{eqnarray*}
due to the facts that
 $p+q<\frac{N+\alpha}{N-2}$ and $s_{n-1}>1$, then we obtain that
 $$
u_{n_0}\in
 \arraycolsep=1pt\left\{
\begin{array}{lll}
 \displaystyle  L^{\frac{Ns_n}{N-2s_n}}(\R^N)\quad
 {\rm if}\quad  s_n<\frac N2,\\[2mm]
 \displaystyle   L ^{t}(\R^N)  \quad {\rm for\ any}\quad t>1 \quad{\rm if}\quad s_n=\frac N2,
 \\[2mm]
  L ^{\infty}(\R^N) \quad
  {\rm if}\quad  s_n>\frac N2.
\end{array}
\right.
$$
For $s_n>\frac N2$, we are done, for $s_n=\frac N2$, we may repeat the above process again to have
$u_{n_0}\in L^\infty(\R^N)$, and then we are done.
For $s_n<\frac N2$, we have that
\begin{eqnarray*}
    s_n
     \ge    \left(  \frac{t_1 -1}{qt_1 }\frac{N  }{N-2s_1} \right)^{n-1} s_1 \to +\infty\quad{\rm as}\; n\to+\infty.
\end{eqnarray*}
Thus, there exists $n_1$  such that $s_{n_1}\ge \frac{N}{2}$ and then
$$u_{n_0}\in L^\infty(\R^N).$$
Therefore,
\begin{equation}\label{4.1.1}
 k\Gamma_0\le u\le u_{n_0} +\sum_{i=1}^{n_0-1}\Gamma_i+k\Gamma_0.
\end{equation}
Note that for $i=1,2,\cdots, n_0-1$,
$$ \Gamma_i(x)\le c_{27} |x|^{T_i},$$
where $T_i>2-N$. As a consequence, we obtain the conclusion.
\end{proof}

\smallskip

\smallskip

\noindent{\bf Proof of Theorem \ref{teo 0}.}
From Proposition \ref{pr 2.1}, we obtain that  $I_\alpha[u^p]u^q\in L^1(\R^N)$ and  $u$ is a  weak solution of (\ref{eq 1.2}) for some $k\ge0$.
Furthermore, if $(p,\,q)$ is supercritical,  we have that $k=0$. For the subcritical case,  we derive  that $u$ is a classical solution of (\ref{eq 1.3}) if $k=0$
from Proposition \ref{pr 2.2}, and  (\ref{1.2}) holds by Proposition \ref{pr 2.3} if $k>0$.
\hfill$\Box$

\smallskip

\vskip0.5cm
\section{Existence}
In this section, we give the proof of Theorem \ref{teo 1}. To this end, denote by
 $\Phi_0$   the solution of
$$  \arraycolsep=1pt
\begin{array}{lll}
 \displaystyle\ \   -\Delta u+ \frac14u=\delta_0\quad
 &{\rm in}\quad  \mathbb{R}^N,\\[2mm]
 \phantom{  }
 \displaystyle   \lim_{|x|\to+\infty}u(x)=0.
\end{array}
$$
By constructing suitable super and sub solution, we derive that
\begin{equation}\label{4.1}
\lim_{|x|\to0^+}\frac{\Gamma_0(x)}{\Phi_0(x)}=1,\qquad  \lim_{|x|\to+\infty}\frac{\Gamma_0(x)}{\Phi_0(x)}=0
\end{equation}
and
\begin{equation}\label{4.2}
 \Gamma_0 \le \Phi_0  \quad {\rm in}\;\; \R^N\setminus\{0\}.
\end{equation}

\begin{proposition}\label{pr 4.1}
Assume that $p>0,\, q\ge1$ satisfy (\ref{1.4}),  then there exists $c_{28}>0$ such that
\begin{equation}\label{4.3}
\mathbb{G}[I_\alpha[\Phi_0^p]\Phi_0^q]\le c_{28}\Phi_0\quad {\rm in}\;\; \R^N\setminus\{0\}.
\end{equation}

\end{proposition}
\begin{proof} Observe that $ \mathbb{G}[I_\alpha[\Phi_0^p]\Phi_0^q]$ is in $C^2(\R^N\setminus\{0\})$ and
has the singularity $|x|^{(2-N)(p+q)+\alpha+2}$ near the origin, which is weaker than $\Phi_0$
by the fact that
$$\lim_{|x|\to0^+}\Phi_0(x)|x|^{N-2}  =c_N.$$
Thus we only need to consider the asymptotic behavior of $\mathbb{G}[I_\alpha[\Phi_0^p]\Phi_0^q]$ at infinity.

Since
$$\lim_{|x|\to+\infty}\Phi_0(x)|x|^{\frac{N-1}2}e^{\frac12|x|} =e^{\frac12} $$
and $\Phi_0$ is radially symmetric and decreasing with respect to $|x|$,
we have that  for $|x|>2$,
\begin{eqnarray*}
 I_\alpha[\Phi_0^p](x) &=& \int_{B_{\frac{|x|}2}(0)} \frac{\Phi_0^p(y)}{|x-y|^{N-\alpha}}dy +\int_{\R^N\setminus B_{\frac{|x|}2}(0)} \frac{\Phi_0^p(y)}{|x-y|^{N-\alpha}}dy
 \\   &\le&  c_{29}\norm{\Phi_0}_{L^p(\R^N)}^p |x|^{\alpha-N}+\Phi_0^{\frac p2}(\frac{|x|}2) \int_{\R^N} \frac{\Phi_0^{\frac p2}(y)}{|x-y|^{N-\alpha}}dy
 \\   &\le&  c_{30}\norm{\Phi_0}_{L^p(\R^N)}^p |x|^{\alpha-N}+c_{33}\Phi_0^{\frac p2}(\frac{|x|}2)
   \\&\le &c_{31}|x|^{\alpha-N},
\end{eqnarray*}
thus, there exists $r>2$ such that
$$\Phi_0 <1,\quad  I_\alpha[\Phi_0^p]\le \frac14\quad {\rm in}\quad \R^N\setminus B_r(0).$$
Moreover, we observe that for $|x|\ge r$,
\begin{eqnarray*}
-\Delta \Phi_0+\Phi_0 =  \frac34 \Phi_0\ge   I_\alpha[\Phi_0^p] \Phi_0^q
\end{eqnarray*}
and
$$\mathbb{G}[I_\alpha[\Phi_0^p] \Phi_0^q]\le c_{32}\Phi_0 \quad {\rm on}\quad \partial B_r(0),$$
then it implies by Comparison Principle  that
$$\mathbb{G}[I_\alpha[\Phi_0^p] \Phi_0^q]\le c_{28}\Phi_0\quad{\rm in}\quad \R^N\setminus B_r(0).$$
This ends the proof. \end{proof}

\bigskip

\noindent {\bf Proof of Theorem \ref{teo 1}.}   Let $v_0:= k \Gamma_0>0$. We define the sequence $\{v_n\}_{n}$ by the iteration
$$
 v_n  =  \mathbb{G}[I_\alpha[v_{n-1}^p] v_{n-1}^q]+ k \Gamma_0, \; n\geq 1.
$$
Observe that
$$v_1= \mathbb{G} [I_\alpha[v_0^p] v_0^q ] + k \Gamma_0>v_0$$
and  assume that
$$
v_{n-1} \ge  v_{n-2} \quad{\rm in} \quad \R^N\setminus\{0\},
$$
 for $n \geq 2$, then we have 
\begin{eqnarray*}
 v_n =   \mathbb{G}[I_\alpha[v_{n-1}^p] v_{n-1}^q]+ k \Gamma_0
 \ge  \mathbb{G}[I_\alpha[v_{n-2}^p] v_{n-2}^q]+ k \Gamma_0
 =   v_{n-1}.
\end{eqnarray*}
Thus, the sequence $\{v_n\}_n$ is increasing with respect to $n$.
Moveover, we have that
\begin{equation}\label{4.2.3}
\int_{\R^N} v_n (-\Delta \xi+\xi) \,dx =\int_{\R^N} I_\alpha[v_{n-1}^p] v_{n-1}^q\xi \,dx +k\xi(0), \quad \forall \xi\in C^\infty_c(\R^N).
\end{equation}

We next build an upper bound for the sequence $\{v_n\}_n$.  For  $t>0$, denote by
$$w_t: =t k^{p+q}\mathbb{G}[I_\alpha[\Phi_0^p]\Phi_0^q]+ k\Phi_0\le (c_{28}t k^{p+q}+ k)\Phi_0,$$
where $c_{28}>0$ is from Proposition \ref{pr 4.1},
 then
\begin{eqnarray*}
  \mathbb{G} [I_\alpha[w_t^p]w_t^q]+k\Phi_0\le  (c_{28}t k^{p+q}+ k)^{p+q}\mathbb{G}[I_\alpha[\Phi_0^p] \Phi_0^q] +  k  \Phi_0
   \le  w_t,
\end{eqnarray*}
if
$$
 (c_{28}t k^{p+q}+ k)^{p+q}\le tk^{p+q},
$$
that is,
\begin{equation}\label{4.2.4}
  (c_{28}t k^{p+q-1} + 1)^{p+q}\le t.
\end{equation}

Note that the convex function $f_{k}(t) =  (c_{28}t k^{p+q-1} + 1)^{p+q}$  can intersect the line $g(t) = t$,  if
\begin{equation}\label{4.2.5}
    c_{28} k^{p+q-1}\le \frac1{p+q}\left(\frac{p+q-1}{p+q}\right)^{p+q-1}.
\end{equation}
Let $k_q=\left(\frac1{c_{28} (p+q)}\right)^{\frac1{p+q-1}}\frac{p+q-1}{p+q}$, then if $k\le k_q$, it always holds that $f_{k}(t_q)\le t_q$ for $ t_q=\left(\frac {p+q}{p+q-1}\right)^{p+q}.$
 Hence
 we have  $w_{t_q}>v_0$ and
$$v_1= \mathbb{G} [I_\alpha[v_0^p]v_0^q]+ k \Gamma_0\le t_q k^{p+q}\mathbb{G} [I_\alpha[\Phi_0^p]\Phi_0^q]+ k \Phi_0=w_{t_q}.$$
Inductively, we obtain
\begin{equation}\label{2.10a}
v_n\le w_{t_q}
\end{equation}
for all $n\in\N$. Therefore the sequence $\{v_n\}_n$ converges to some function $u_{k}$.
By \eqref{4.2.3}, $u_{k}$ is a weak solution of (\ref{eq 1.2})
and satisfies (\ref{1.2}).

For $k\le k_q$, we have that $u_k\le w_{t_q}$ in $\R^N\setminus\{0\}$, so  $u_k\in L^p(\R^N)$ and $I_\alpha[u_k^p]u_k^q$ is bounded uniformly locally in $\R^N\setminus\{0\}$,
then $u_k$ is a classical solution of (\ref{eq 1.1}).

We claim that $u_{k}$ is the minimal solution of (\ref{eq 1.1}), that is, for any positive solution $u$ of (\ref{eq 1.2}), we always have $u_{k}\leq u$. Indeed,  there holds
\[
 u  = \mathbb{G}[ I_\alpha[ u^p]u^q]+ k \Gamma_0\ge v_0,
\]
and then
\[
 u  = \mathbb{G}[ I_\alpha[ u^p]u^q]+ k \Gamma_0\ge  \mathbb{G}[ I_\alpha[ v_0^p]v_0^q]+ k \Gamma_0=v_{1}.
\]
We may show inductively that
\[
u\ge v_n
\]
for all $n\in\N$.  The claim follows.

 Similarly,  if problem (\ref{eq 1.2}) has a positive solution $u$  for some $ k_1>0$, then (\ref{eq 1.2}) admits a minimal solution $u_{k}$ for all $ k\in(0, k_1]$. As a result, the mapping $ k\mapsto u_{k}$ is increasing.
So we may define
$$k^*=\sup\{k>0:\ (\ref{eq 1.2})\ {\rm has\ minimal\ positive \ solution\ for\ }k \},$$
then $k^*\ge k_q.$

We next prove that $k^*<+\infty$.

Let $\eta_0$ be a radially symmetric $C^\infty_c$-function such that $\eta_0=0$ in $\R^N\setminus B_2(0)$ and $\eta_0=1$ in $B_1(0)$.
For $\epsilon\in(0,1)$,  denote
 $$\xi_\epsilon(x)=\eta_0(\epsilon x)\mathbb{G}[\eta_0](x),\quad x\in\R^N.$$
By the direct computation, we have that
\begin{eqnarray*}
-\Delta \xi_\epsilon(x) + \xi_\epsilon(x) =\epsilon^2 (-\Delta)\eta_0(\epsilon x) \mathbb{G}[\eta_0](x) + 2\epsilon  \nabla  \eta_0(\epsilon x) \cdot \nabla \mathbb{G}[\eta_0](x)+\eta_0 (\epsilon x)\eta_0(x).
\end{eqnarray*}
Choosing $\epsilon>0$ small, we deduce that
\begin{eqnarray*}
 \int_{\R^N} u_{k}\big(-\Delta \xi_\epsilon+ \xi_\epsilon \big)\,dx & =&\int_{\R^N}u_k(x)  [ -\epsilon^2 \Delta\eta_0(\epsilon x)\mathbb{G}[\eta_0](x) + 2\epsilon  \nabla  \eta_0(\epsilon x) \cdot \nabla \mathbb{G}[\eta_0](x)\\&&+\eta_0 (\epsilon x)\eta_0(x)]\,dx
\\&\le&  \int_{B_2(0)} u_k(x)\,dx + c_{33}(\epsilon+\epsilon^2)\int_{B_{\frac2\epsilon}(0)\setminus B_{\frac1\epsilon}(0)} u_k(x)\, dx
\\&\le& 2 \int_{B_2(0)} u_k(x)\,dx,
\end{eqnarray*}
where we have used $ess {\rm{sup}}_{\R^N\setminus B_{\frac1\epsilon}(0)} u_k(x)\to 0$ as $\epsilon\to0$.

Since
$$u_k\ge k\Gamma_0\ge c_{34}k \quad{\rm in}\quad B_1(0)$$
and
$$I_\alpha[u_k^p]\ge c_{35}k^p \quad{\rm in}\quad B_1(0)$$
for some $c_{40}>0$ independent of $k$, then
\begin{eqnarray*}
2 \int_{B_2(0)} u_k(x)\,dx  &\ge&    \int_{\R^N} u_{k}[(-\Delta) \xi_\epsilon+\xi_\epsilon]\,dx =\int_{\R^N} I_\alpha[u_{k}^p]u_k^q\xi_\epsilon\, dx
\\& \ge &  c_{36} k^{p+q-1}\int_{B_2(0)} u_k(x)\,dx,
\end{eqnarray*}
where $p+q>1$.
Thus,
$$
 k\le  c_{37},
$$
so does $k^*$ which ends the proof. \hfill$\Box$

\begin{remark}
Concerning the existence of weak solutions of (\ref{eq 1.2}) for $0<k<k^*$ in the subcritical case, we may consider the stability of the minimal solution $u_k$ and
then construct the second solution by using Mountain Pass Theorem \cite[Theorem~6.1]{struwe}.

\end{remark}

\section{Appendix}

It is well-known that  the Green kernel $G(x,y)$  of $-\Delta+id$ in $\R^N\times\R^N$ is $\Gamma_0(x-y)$, which has exponential decay at infinity.
We recall that $\mathbb{G}[\cdot]$  the
Green operator defined as
$$\mathbb{G}[f](x)=\int_{\R^N} G(x,y)f(y)dy. $$

\begin{proposition}\label{embedding}\cite[Lemma A.3]{NS}
Assume that   $h\in L^s(\R^N)$ with $s\ge1$, then

 \noindent$(i)$
\begin{equation}\label{a 4.1}
\|\mathbb{G}[h]\|_{L^\infty(\R^N)}\le c_{38}\|h\|_{L^s(\R^N)}\quad{\rm if}\quad \frac1s<\frac{ 2 }N;
\end{equation}

\noindent$(ii)$
\begin{equation}\label{a 4.2}
\|\mathbb{G}[h]\|_{L^r(\R^N)}\le c_{38}\|h\|_{L^s(\R^N)}\quad{\rm if}\quad \frac1s\le \frac1r+\frac{2}N\quad
\rm{and}\quad s>1;
\end{equation}

\noindent$(iii)$
\begin{equation}\label{a 4.02}
\|\mathbb{G}[h]\|_{L^r(\R^N)}\le c_{38}\|h\|_{L^1(\R^N)}\quad{\rm if}\quad 1<\frac1r+\frac{2}N.
\end{equation}
\end{proposition}

Recall that
$$I_\alpha[h](x)=\int_{\R^N}\frac{h(y)}{|x-y|^{N-\alpha}}dy\quad{\rm for}\ \ h\in L^1(\R^N).$$

\begin{proposition}\label{embedding1}
Suppose that $\alpha\in(0,N)$, $\Omega\subset B_{R/2}(0)$ for some $R>0$   and $h\in L^s(B_R(0))\cap L^1(\R^N)$ for some $s\ge1$. Then
\begin{equation}\label{b 4.1}
\|I_\alpha[h]\|_{L^\infty(\Omega)}\le c_{39}\left(\|h\|_{L^s(B_R(0))}+\|h\|_{L^1(\R^N)}\right)\quad{\rm if}\quad \frac1s<\frac{\alpha }N,
\end{equation}
\begin{equation}\label{b 4.2}
\| I_\alpha[h]\|_{L^r(\Omega)}\le c_{39}\left(\|h\|_{L^s(B_R(0))}+\|h\|_{L^1(\R^N)}\right)\quad{\rm if}\quad \frac1s\le \frac1r+\frac{\alpha}N\quad
 {\rm and}\quad s>1
\end{equation}
and
\begin{equation}\label{b 4.02}
\| I_\alpha[h]\|_{L^r(\Omega)}\le c_{39}  \|h\|_{L^1(\R^N)} \quad{\rm if}\quad 1<\frac1r+\frac{\alpha}N.
\end{equation}
\end{proposition}

\begin{proof}
Observe that  $|x-y|> \frac{R}{2}$ for $x\in \Omega$ and $y\in \R^N\setminus B_{R}(0)$ and then for $x\in\Omega$,
\begin{eqnarray*}
|I_\alpha[h](x)| &\le & \int_{B_{R}(0)}\frac{|h(y)|}{|x-y|^{N-\alpha}}dy +  \int_{\R^N\setminus B_{R}(0)} \frac{|h(y)|}{|x-y|^{N-\alpha}}dy
 \\&\le& \int_{B_{R}(0)}\frac{|h(y)|}{|x-y|^{N-\alpha}}dy+ (R/2)^{\alpha-N} \norm{h}_{L^1(\R^N)}.
\end{eqnarray*}
Without loss of generality, we can assume
$h\ge 0$ and in the following, we only need to consider
$$J_\alpha[h](x):=\int_{B_{R}(0)}\frac{h(y)}{|x-y|^{N-\alpha}}dy,\quad x\in\Omega.$$

 First we prove (\ref{b 4.1}). By H\"{o}lder's inequality,
\begin{eqnarray*}
  J_\alpha[h](x)\
 &\le&  \big(\int_{B_R(0)} \frac1{|x-y|^{(N-\alpha)s'}}dy\big)^{\frac1{s'}} \big(\int_{B_R(0)}|h(y)|^sdy\big)^\frac1s
 \\&\le&c_{40}\|h\|_{L^s(B_R(0))} \int_{B_R(0)}
 \frac1{|x-y|^{(N-\alpha)s'}}dy,
\end{eqnarray*}
where $s'=\frac s{s-1}$. Since $\frac1s<\frac{2}N$ and $(N-\alpha)s'<N$, we have
\begin{eqnarray*}
\int_{B_R(0)}
 \frac1{|x-y|^{(N-\alpha)s'}}dy
 =c_{41}\int_0^R r^{N-1-(N-\alpha)s'}dr
 \le c_{42} R^{N-(N-\alpha)s'},
\end{eqnarray*}
 then  (\ref{b 4.1}) holds.

Next, we prove (\ref{b 4.2}) for $r\le s$ and (\ref{b 4.02}) for
$r = 1$. There holds
\begin{eqnarray*}
 \left[\int_{B_R(0)}J_\alpha[h]^r(x)dx\right]^{\frac1r}&\le& \left\{\int_{\R^N}\bigg[\int_{\R^N}\frac{h(y)\chi_{B_R(0)}(x)\chi_{B_R(0)}(y)}{|x-y|^{N-\alpha}}dy\bigg]^rdx\right\}^{\frac1r}
 \\&=& \left\{\int_{\R^N}\left[\int_{\R^N}\frac{h(x-y)\chi_{B_R(0)}(x)\chi_{B_R(0)}(x-y)}{|y|^{N-\alpha}}dy\right]^rdx\right\}^{\frac1r}.
\end{eqnarray*}
By  Minkowski's inequality, we have that
\begin{eqnarray*}
  \left[\int_{B_R(0)}J_\alpha[h]^r(x)dx\right]^{\frac1r}&\le&
 \int_{\R^N}\bigg[\int_{\R^N}\frac{h^r(x-y)\chi_{B_R(0)}(x)\chi_{B_R(0)}(x-y)}{|y|^{(N-\alpha)r}}dx\bigg]^{\frac1r}dy
 \\&\le &  \int_{B_{2R}(0)}\bigg[\int_{\R^N} h^r(x-y)\chi_{B_R(0)}(x)\chi_{B_R(0)}(x-y)dx\bigg]^{\frac1r}\frac1{|y|^{N-\alpha}}dy
\\&\le &  \|h\|_{L^r(B_R(0))}\int_{B_{2R}(0)}\frac1{|y|^{N-\alpha}}dy
\\&\le& c_{43}\|h\|_{L^s(B_R(0))}.
\end{eqnarray*}

Finally, we prove (\ref{b 4.2}) in the case
$r> s\ge1$ and $\frac1s\le \frac1r+\frac{2}N$,  and (\ref{b 4.02}) for $r> 1$ and $1< \frac1r+\frac{\alpha}N.$
We claim that if $r>s$ and $\frac1{r^*}=\frac1{s}-\frac{\alpha}N$, the mapping
$h\to  J_\alpha(h)$ is of weak-type $(s,r^*)$ in the sense that
\begin{equation}\label{weak rs-1}
|\{x\in\Omega: | J_\alpha[h](x)|>t\}|\le
\big(A_{s,r^*}\frac{\|h\|_{L^{s}(B_R(0))}}{t}\big)^{r^*},\quad h\in
L^s(B_R(0)),\ \ \forall t>0,
\end{equation}
where $A_{s,r^*}$ is a positive constant. Defining
$$
J_0(x,y)=\left\{ \arraycolsep=1pt
\begin{array}{lll}
|x-y|^{\alpha-N}\quad & {\rm if}\quad |x-y|\le \nu,\\[2mm]
0\quad & {\rm if}\quad |x-y|> \nu
\end{array}
\right.
$$
 for $\nu>0$ and  $J_\infty(x,y)=|x-y|^{\alpha-N}-J_0(x,y)$. Then we have that
$$|\{x\in\Omega:   J_\alpha[h](x) >2t\}|\le |\{x\in\Omega:  J_0[h](x) >t\}|+|\{x\in\Omega:  J_\infty[h](x) >t\}|,$$
where $J_0[h]=\int_{B_R(0)} J_0(x,y) h(y)\,dy$ and $J_\infty[h]=\int_{B_R(0)} J_\infty(x,y) h(y)\,dy$.
By Minkowski's inequality, we obtain that
\begin{eqnarray*}
|\{x\in B_R(0): |J_0[h](x)|>t\}|&\le&
\frac{\|J_0(h)\|^s_{L^s(B_R(0))}}{t^s}
\\&\le &\frac{\|\int_{B_R(0)} \chi_{B_\nu(x-y)}|x-y|^{\alpha-N}|h(y)|dy\|^s_{L^s(B_R(0))}}{t^s}
\\&\le&\frac{[\int_{B_R(0)}(\int_{B_R(0)} |h(x-y)|^sdx)^{\frac1s}|y|^{\alpha-N}\chi_{B_\nu}(y)dy]^s}{t^s}
\\&\le&\frac{\|h\|^s_{L^s(B_R(0))} \int_{B_\nu}|x|^{-N+\alpha}dx }{t^s}=c_{44}\frac{\|h\|^s_{L^s(B_R(0))} \nu^{\alpha}  }{t^s}.
\end{eqnarray*}
On the other hand,
\begin{eqnarray*}
\|J_\infty[h]\|_{L^\infty(B_R(0))}&\le&\|\int_{B_R(0)}
\chi_{B_\nu^c}(x-y)|x-y|^{\alpha-N}|h(y)|dy\|_{L^\infty(B_R(0))}
\\&\le&(\int_{B_R(0)} |h(y)|^sdy)^{\frac1s}\|(\int_{\Omega\setminus B_\nu(y)}|x-y|^{(\alpha-N)s'}dy)^{\frac1{s'}}\|_{L^\infty(B_R(0))}
\\&\le&c_{45}\|h\|_{L^s(B_R(0))}\nu^{\alpha-\frac
Ns},
\end{eqnarray*}
where $s'=\frac s{s-1}$ if $s>1$, and if $s = 1$, $s'=\infty$. Choosing $\nu=(\frac
t{c_{45}\|h\|_{L^s(B_R(0))}})^{\frac1{\alpha-\frac Ns}}$, we obtain
$$\|J_\infty[h]\|_{L^\infty(B_R(0))}\le t,$$
which means that
$$|\{x\in\Omega: |J_\infty[h](x)|>t\}|=0.$$
With this choice of $\nu$, we have that
$$|\{x\in\Omega: |J_\alpha[h]|>2t\}|\le c_{46}\frac{\|h\|^s_{L^s(B_R(0))}\nu^{\alpha s}}{t^s}
\le c_{47}\left(\frac{\|h\|_{L^s(B_R(0))}}{t}\right)^{r^*}.$$ The claim
for $r>s$ follows from the Marcinkiewicz
interpolation theorem.
\end{proof}

\thanks{Acknowledgements: The research of the first
author is supported by NSFC (11401270). The research of the second author is partially supported by NSFC (11271133 and 11431005) and
and Shanghai Key Laboratory of PMMP.}


\end{document}